\documentclass{amsart}
\usepackage{graphicx}

\newtheorem{theorem}{Theorem}[section]

\theoremstyle{definition}

\newtheorem{proposition}[theorem]{Proposition}

\theoremstyle{remark}

\numberwithin{equation}{section}

\begin{document}

\title{Chebyshev's bias and generalized \\Riemann hypothesis}

\author{Adel Alamadhi$^*$}
\address{$*$MECAA, King Abdulaziz University, Jeddah, Saudi Arabia.}

\author{Michel Planat$\dag$}
\address{$\dag$Institut FEMTO-ST, CNRS, 32 Avenue de l'Observatoire, F-25044 Besan\c con, France.}
%\curraddr{Department of Mathematics and Statistics,
%Case Western Reserve University, Cleveland, Ohio 43403}
\email{michel.planat@femto-st.fr}

\author{Patrick Sol\'e$\ddag,^*$}
\address{$\ddag$Telecom ParisTech, 46 rue Barrault, 75634 Paris Cedex 13, France.}
%\address{$^*$MECAA, King Abdulaziz University, Jeddah, Saudi Arabia.}
\email{sole@telecom-paristech.fr}

%\author{Adel Alamadhi$^*$}
%\address{$*$MECAA, King Abdulaziz University, Jeddah, Saudi Arabia.}

%    General info

\date{December 8\normalfont, 2011 and, in revised form}



\begin{abstract}

It is well known that $\mbox{li}(x)>\pi(x)$ (i) up to the (very large) Skewes' number $x_1 \sim 1.40 \times 10^{316}$ \cite{Bays00}. But, according to a Littlewood's theorem, there exist infinitely 
many $x$ that violate the inequality, due to the specific distribution of non-trivial zeros $\gamma$ of the Riemann zeta function $\zeta(s)$, encoded by the equation 
$\mbox{li}(x)-\pi(x)\approx \frac{\sqrt{x}}{\log x}[1+2 \sum_{\gamma}\frac{\sin (\gamma \log x)}{\gamma}]$ (1). If Riemann hypothesis (RH) holds, (i) may be replaced by the equivalent statement $\mbox{li}[\psi(x)]>\pi(x)$ (ii) due to Robin \cite{Robin84}.
A statement similar to (i) was found by Chebyshev that $\pi(x;4,3)-\pi(x;4,1)>0$ (iii) holds for any $x<26861$ \cite{Rubin94} (the notation $\pi(x;k,l)$ means the number of primes  up to $x$ and congruent to $l\mod k$).
The {\it Chebyshev's bias }(iii) is related to the generalized Riemann hypothesis (GRH) and occurs with a logarithmic density $\approx 0.9959$ \cite{Rubin94}.
In this paper, we reformulate the Chebyshev's bias for a general modulus $q$ as the inequality
 $B(x;q,R)-B(x;q,N)>0$ (iv), where $B(x;k,l)=\mbox{li}[\phi(k)*\psi(x;k,l)]-\phi(k)*\pi(x;k,l)$ is a counting function introduced in Robin's paper \cite{Robin84} and $R$( resp. $N$) is a quadratic residue modulo $q$ (resp. a non-quadratic residue). We investigate numerically the case $q=4$ and a few prime moduli $p$. Then, we proove that (iv) is equivalent to GRH for the modulus $q$.

%   The proof follows from the use of a formula similat to (1), involving the non-trivial zeros of the Dirichlet function $L(s)=\sum_{n\ge 0}\frac{(-1)^n}{(2n+1)^s}$. 
%Chebyshev's bias exist for moduli different from $4$ and may be addressed in a similar way under GRH with the corresponding $L$-functions.

\end{abstract}

\maketitle

\section{Introduction}

In the following, we denote $\pi(x)$ the prime counting function and $\pi(x;q,a)$ the number of primes not exceeding $x$ and congruent to $a\mod q$. The asymptotic law for the distribution of primes is the prime number theorem
$\pi(x)\sim \frac{x}{\log x}$. Correspondingly, one gets \cite[eq. (14), p. 125]{Davenport80} 
\begin{equation}
\pi(x;q,a)\sim \frac{\pi(x)}{\phi(q)}
\label{eq1}
\end{equation}
that is, one expects the same number of primes in each residue class $a \mod q$, if $(a,q)=1$. Chebyshev's bias is the observation that, contrarily to expectations, $\pi(x;q,N)>\pi(x;q,R)$ most of the times, when $N$ is a non-square modulo $q$, but $R$ is. 

Let us start with the bias
\begin{equation}
\delta(x,4):=\pi(x;4,3)-\pi(x;4,1)
\label{eq2}
\end{equation}
found between the number of primes in the non-quadratic residue class $N=3 \mod 4$ and the number of primes in the quadratic one $R=3 \mod 4$. The values $\delta(10^n,4)$, $n\le 1$, form the increasing sequence 
$$A091295=\{1,~2,~7,~10,~25,~147,~218,~446,~551,~5960,\ldots\}.$$
The bias is found to be negative in thin zones of size 
$$\{2,~410,~15~358,~41346,~42~233~786,~416~889~978,\ldots\}$$
spread over the location of primes of maximum negative bias \cite{Bays1979}
$$\{26861,~623~681,~12~366~589,~951~867~937,~6~345~026~833,~18~699~356~321\ldots\}.$$

It has been proved that there are are infinitely many sign changes in the Chebyshev's bias (\ref{eq2}). This follows from the Littlewood's oscillation theorem \cite{Deleglise2000,I}
\begin{equation}
\delta(x,4):=\Omega_{\pm}\left(\frac{x^{1/2}}{\log x}\log_3x\right).
\label{eq3}
\end{equation}
A useful measure of the Chebyshev's bias is the logarithmic density \cite{Rubin94,Deleglise2000,Fiorilli2011}
\begin{equation}
d(A)=\lim_{x \rightarrow \infty} \frac{1}{\log x}\sum_{a\in A, a\le x}\frac{1}{a}
\label{eq4}
\end{equation}
for the positive $\Delta^+$ and negative $\Delta^-$ regions calculated as $d(\Delta^+)=0.9959$ and $d(\Delta^-)=0.0041$.

In essence, Chebyshev's bias $\delta(x,4)$ is similar to the bias
\begin{equation}
\delta(x):=\mbox{Li}(x)-\pi(x).
\label{eq5}
\end{equation}
It is known that $\delta(x)>0$ up to the (very large) Skewes' number $x_1\sim 1.40 \times 10^{316}$ but, according to Littlewood's theorem, there also are infinitely many sign changes of $\delta(x)$ \cite{I}.

The reason why the asymmetry in (\ref{eq5}) is so much pronounced is encoded in the following approximation of the bias \cite{Rubin94,Bays2001}\footnote{The bias may also be approached in a different way by relating it to the second order Landau-Ramanujan constant \cite{Moree2003}.}
\begin{equation}
\delta(x)\sim\frac{\sqrt{x}}{\log x} \left(1+2\sum_{\gamma}\frac{\sin(\gamma \log x+\alpha_{\gamma})}{\sqrt{1/4+\gamma^2}}\right),
\label{eq6}
\end{equation}
where $\alpha_{\gamma}=\cot^{-1}(2\gamma)$ and $\gamma$ is the imaginary part of the non-trivial zeros of the Riemann zeta function $\zeta(s)$. The smallest value of $\gamma$ is quite large, $\gamma_1\sim 14.134$, and leads to a large asymmetry in (\ref{eq5}).

Under the assumption that the generalized Riemann hypothesis (GRH) holds that is, if the Dirichlet L-function with non trivial real character $\kappa_4$
\begin{equation}
L(s,\kappa_4)=\sum_{n\ge 0}\frac{(-1)^n}{(2n+1)^s},
\label{eq7}
\end{equation}
has all its non-trivial zeros located on the vertical axis $\Re(s)=\frac{1}{2}$, then the formula (\ref{eq6}) also holds for the Chebyshev's bias $\delta(x,4)$. The smallest non-trivial zero of $L(s,\kappa_4)$ is at $\gamma_1\sim 6.02$, a much smaller value than than the one corresponding to $\zeta(s)$, so that the bias is also much smaller.

A second factor controls the aforementionned assymmetry of a $L$-function of real non-trivial character $\kappa$, it is  the {\it variance} \cite{Martin2000}
\begin{equation}
V(\kappa)=\sum_{\gamma>0}\frac{2}{1/4+\gamma^2}.
\label{eq8}
\end{equation}
For the function $\zeta(s)$ and $L(s,\kappa_4)$ one gets $V=0.045$ and V=0.155, respectively. 

\subsection*{Our main goal}

In a groundbreaking paper, Robin reformulated the unconditional bias (\ref{eq5}) as a conditional one involving the second Chebyshev function $\psi(x)=\sum_{p^k\le x} \log p$
\begin{equation}
~\mbox{The}~\mbox{equality}~\delta'(x):=\mbox{li}[\psi(x)]-\pi(x)>0 ~\mbox{is}~\mbox{equivalent}~\mbox{to}~\mbox{RH}.
\label{eq9}
\end{equation}
This statement is given as Corollary 1.2 in \cite{PlanatSole2011} and led the second and third author of the present work  to derive a {\it good prime counting function} 
\begin{equation}
\pi(x)=\sum_{n=1}^3 \mu(n) \mbox{li}[\psi(x)^{1/n}].
\label{eq10}
\end{equation}

Here, we are interested in a similar method to {\it regularize} the Chebyshev's bias in a conditional way similar to (\ref{eq9}). In \cite{Robin84}, Robin introduced the function
\begin{equation}
B(x;q,a)=\mbox{li}[\phi(q)\psi(x;q,a)]-\phi(q)\pi(x;q,a),
\label{eq11}
\end{equation}
that generalizes (\ref{eq9}) and applies it to the residue class $a \mod q$, with $\psi(x,q,a)$ the generalized second Chebyshev's function. Under GRH  he proved that \cite[Lemma 2, p. 265]{Robin84}
\begin{equation}
B(x;q,a)=\Omega_{\pm} \left(\frac{\sqrt{x}}{\log^2 x}\right),
\label{eq12}
\end{equation}
that is 
\begin{equation}
\mbox{The} ~\mbox{inequality}~B(x;q,a)>0 ~\mbox{is}~\mbox{equivalent}~\mbox{to}~\mbox{GRH}.
\label{eq13}
\end{equation}

For the Chebyshev's bias, we now need a proposition taking into account two residue classes such that  $a=N$(a non-quadratic residue) and $a=R$ (a quadratic one).
\begin{proposition}\label{Chebbias}
Let $B(x;q,a)$ be the Robin $B$-function defined in (\ref{eq11}), and $R$ (resp. $N$) be a quadratic residue modulo $q$ (resp. a non-quadratic residue), then 
the statement $\delta'(x,q):=B(x;q,R)-B(x;q,N)>0,~ \forall x$ (i), is equivalent to GRH for the modulus $q$.
\end{proposition}

The present paper deals about the numerical justification of proposition \ref{Chebbias} in Sec. \ref{regularizedCB} and its proof in Sec. \ref{proofCB}. The calculations are performed with the software Magma \cite{Magma} available on a $96$ MB segment of the cluster at the University of Franche-Comt\'e.

\section{The regularized Chebyshev's bias}
\label{regularizedCB}

All over this section, we are interested in the prime champions of the Chebyshev's bias $\delta(x,q)$ (as defined in (\ref{eq2}) or (\ref{eq14}), depending on the context). We separate the prime champions leading to a positive/negative bias. Thus, the $n$-th prime champion satisfies
\begin{equation}
\delta^{(\epsilon)}(x_n,q)=\epsilon n,~\epsilon=\pm 1.
\label{bias}
\end{equation}
We also introduce a new measure of the overall bias $b(q)$, dedicated to our plots, as follows
\begin{equation}
b(q)=\sum_{n,\epsilon}\frac{\delta^{(\epsilon)}(x_n,q)}{x_n}.
\label{biassum}
\end{equation}
Indeed, smaller is the bias lower is the value of $b(q)$.
Anticipating over the results presented below, Table \ref{table1} summarize the calculations.

\begin{table}[ht]
\caption{The new bias (\ref{biassum}) (column 2) and the standard logarithmic density (\ref{eq4}) (column 3). }\label{table1}
\renewcommand\arraystretch{1.5}
\noindent\[
\begin{array}{|c|c|c|c|c|}
\hline
\mbox{modulus}~ q  & \mbox{bias}~ b(q) & \mbox{log}~\mbox{density}~d(\Delta^+)& \mbox{first}~\mbox{zero}~\gamma_1\\
%$N$  & $x$~~max & $\eta$ max & $x~$ max & $(R-\pi)$ max\\
\hline
4 & 0.7926 & 0.9959~ \mbox{\cite{Rubin94}}&14.134\\
11 & 0.1841 & 0.9167~ \mbox{\cite{Rubin94}}&0.2029\\
13 & 0.2803 & 0.9443~ \mbox{\cite{Rubin94}}&3.119\\
163 & 0.0809 & 0.55 ~\mbox{\cite{Bays2001}}&2.477 \\
\hline
%8& $>$ 500, 000&$>10^{3199194}$\\\hline
\end{array}
\]
\end{table}
\begin{figure}
\centering
\includegraphics[]{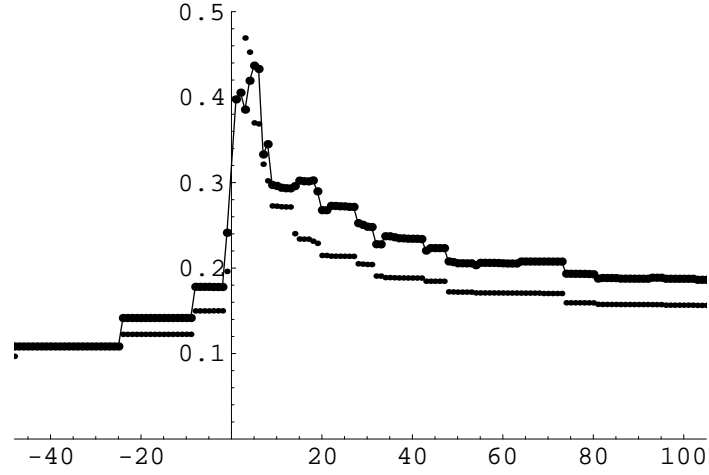}
\caption{The normalized regularized bias $\delta'(x,4)/\sqrt{x}$ versus the Chebyshev's bias $\delta(x,4)$ at the prime champions of $\delta(x,4)$ (when $\delta(x,4)>0$) and at the prime champions of $-\delta(x,4)$ (when $\delta(x,4)<0$). The extremal prime champions in the plot are $x=359327$ (with $\delta=105$) and $x=951867937$ (with $\delta=-48$). The curve is asymmetric around the vertical axis, a fact that reflects the asymmetry of the Chebyshev's bias. As explained in the text, a violation of GRH would imply a negative value of the regularized bias $\delta'(x,4)$. The small dot curve corresponds to the fit of $\delta'(x,4)/\sqrt{x}$ by $2/\log x$ calculated in Sec. 3.}
\label{fig1}
\end{figure}

\subsection*{Chebyshev's bias for the modulus $q=4$}

As explained in the introduction, our goal in this paper is to reexpress a standard Chebyshev's bias $\delta(x,q)$ into a regularized one $\delta'(x,q)$, that is always positive under the condition that GRH holds. Indeed we do not discover any numerical violation of GRH and we always obtains a positive $\delta'(x,q)$. The asymmetry of Chebyshev's bias arises in the plot $\delta$ vs $\delta'$, where the fall of the normalized bias $\frac{\delta}{\sqrt{x}}$ is faster for negative values of $\delta$ than for positive ones. Fig. \ref{fig1} clarifies this effect for the historic modulus $q=4$. We restricted our plot to the champions of the bias $\delta$ and separated positive and negative champions.

\subsection*{Chebyshev's bias for a prime modulus $p$}

For a prime modulus $p$, we define the bias so as to obtain an averaging over all differences $\pi(x;p,N)-\pi(x;p,R)$, where as above $N$ and $R$ denote a non-quadratic and a quadratic residue, respectively
\begin{equation}
\delta(x,p)=-\sum_a \left(\frac{a}{p}\right)\pi(x;p,a),
\label{eq14}
\end{equation}
where $\left(\frac{a}{p}\right)$ is the Legendre symbol. Correspondingly, we define the regularized bias as
\begin{equation}
\delta'(x,p)=\frac{1}{\left\lfloor p/2\right\rfloor} \sum_a \left(\frac{a}{p}\right) B(x;p,a).
\label{eq15}
\end{equation}
\begin{proposition}\label{Chebbiasp}
Let $p$ be a selected prime modulus and $\delta'(x,p)$ as in (\ref{eq15}) then
the statement $\delta'(x,p)>0,~ \forall x $,  is equivalent to GRH for the modulus $p$.
\end{proposition}

As mentioned in the introduction, the Chebyshev's bias is much influenced by the location of the first non-trivial zero of the function $L(s,\kappa_q)$, $\kappa_q$ being the real non-principal character modulo $q$. This is especially true for $L(s,\kappa_{163})$ with its smaller non-trivial zero at $\gamma\sim 0.2029$ \cite{Bays2001}. The first negative values occur at $\{15073,15077,15083,\ldots\}$.

Fig. \ref{fig2} represents the Chebyshev's bias $\delta'$ for the modulus $q=163$ versus the standard one $\delta$ (thick dots). Tha asymmetry of the Chebyshev's bias is revealed at small values of $|\delta|$ where the the fit of the regularized bias by the curve $2/\log x$ is not good (thin dots). 

\begin{figure}
\centering
\includegraphics[]{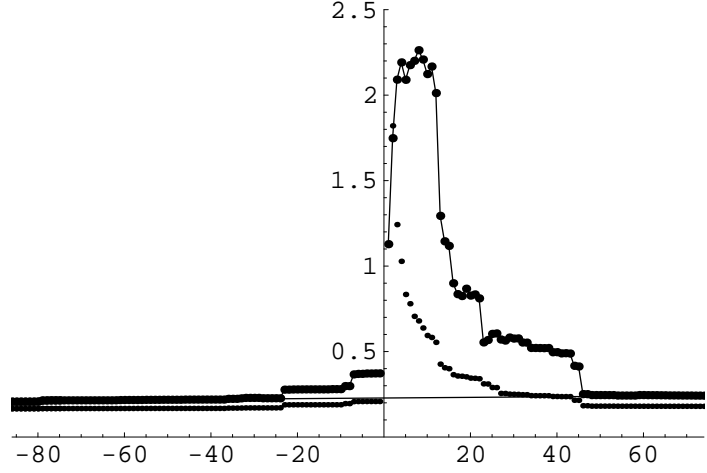}
\caption{The normalized regularized bias $\delta'(x,163)/\sqrt{x}$ versus the Chebyshev's bias $\delta(x,163)$ at all the prime champions of $|\delta(x,163)|$ [from $|\delta(x,163)|>74$ the bias is $\delta(x,163)<0$ negative], superimposed to the curve at the prime champions of $-\delta(x,163)$ (when $\delta(x,163)<0$). The extremal prime champions in the plot are $x=68491$ (with $\delta=74$) and $x=174637$ (with $\delta=-86$).  The asymmetry is still clearly visible in the range of small values of $|\delta|$ but tends to disappear in the range of high values of $|\delta|$. The small dot curve corresponds to the fit of $\delta'(x,163)/\sqrt{x}$ by $2/\log x$ calculated in Sec. 3.}
\label{fig2}
\end{figure}

For the modulus $q=13$, the imaginary part of the first zero is not especially small, $\gamma_1\sim 3.119$, but the variance (\ref{eq8}) is quite high, $V(\kappa_{-13})\sim 0.396$. The first negative values of $\delta(x,13)$ at primes occur when $\{2083,2089,10531,\ldots\}$. Fig. \ref{fig3} represents the Chebyshev's bias $\delta'$ for the modulus $q=13$ versus the standard one $\delta$ (thick dots) as compared to the fit by $2/\log x$ (thin dots). 

Finally, for the modulus $q=11$, the imaginary part of the first zero is quite small, $\gamma_1\sim 0.209$, and the variance is high, $V(\kappa_{-11})\sim 0.507$. In such a case, as shown in Fig. \ref{fig4}, the approximation of the regularized bias by $2/\log x$ is good in the whole range of values of $x$.

\begin{figure}
\centering
\includegraphics[]{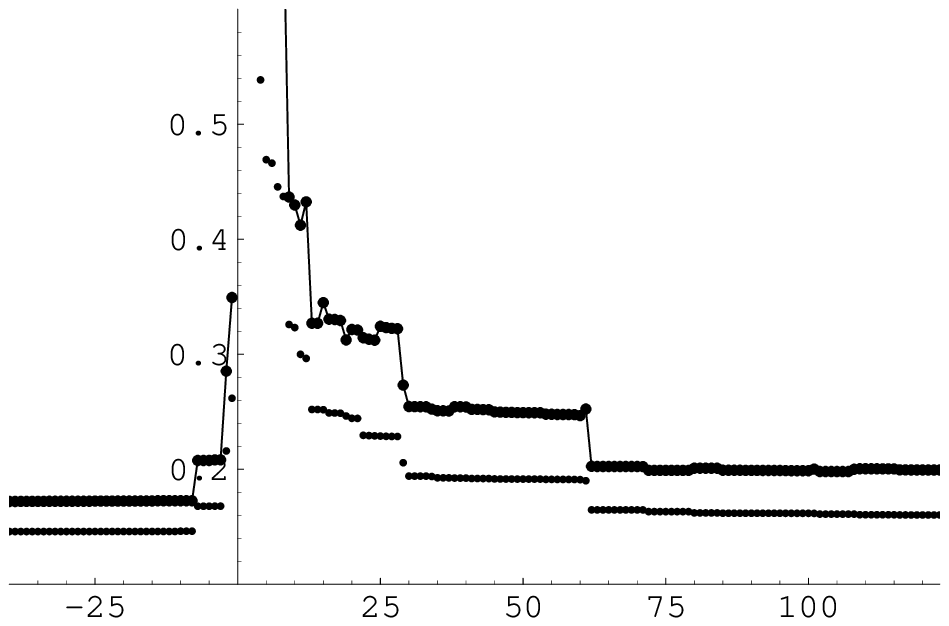}
\caption{The normalized regularized bias $\delta'(x,13)/\sqrt{x}$ versus the Chebyshev's bias $\delta(x,13)$ at the prime champions of $\delta(x,13)$ (when $\delta(x,13)>0$), and the curve at the prime champions of $-\delta(x,13)$ (when $\delta(x,13)<0$). The extremal prime champions in the plot are $x=263881$ (with $\delta=123$) and $x=905761$ (with $\delta=-40$). The small dot curve corresponds to the fit of $\delta'(x,13)/\sqrt{x}$ by $2/\log x$ calculated in Sec. 3.}
\label{fig3}
\end{figure}
\begin{figure}
\centering
\includegraphics[]{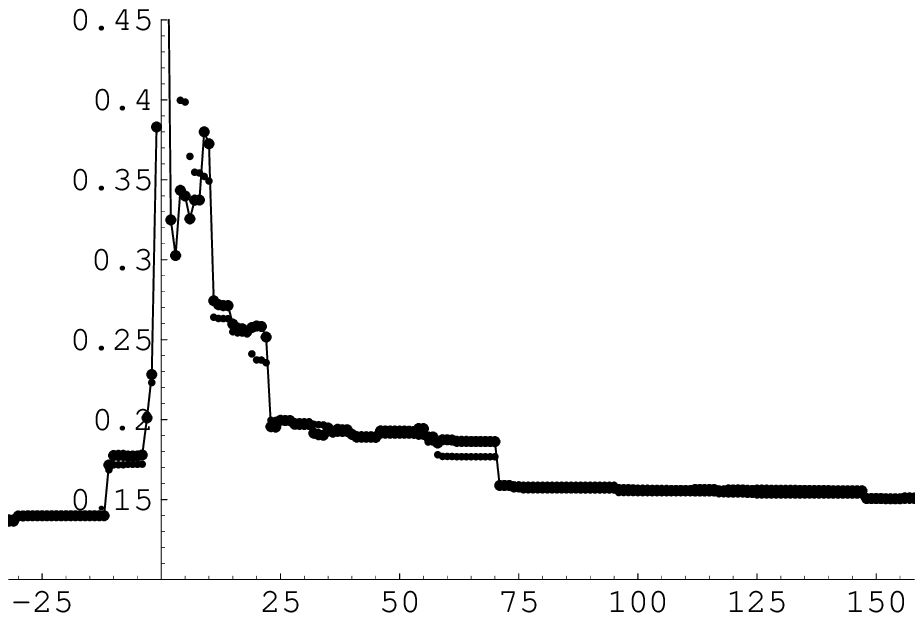}
\caption{The normalized regularized bias $\delta'(x,11)/\sqrt{x}$ versus the Chebyshev's bias $\delta(x,11)$ at the prime champions of $\delta(x,11)$ (when $\delta(x,11)>0$), and the curve at the prime champions of $-\delta(x,11)$ (when $\delta(x,11)<0$). The extremal prime champions in the plot are $x=638567$ (with $\delta=158$) and $x=1867321$ (with $\delta=-32$).The small dot curve corresponds to the (very good) fit of $\delta'(x,11)/\sqrt{x}$ by $2/\log x$ calculated in Sec. 3.}
\label{fig4}
\end{figure}

\section{Proof of proposition \ref{Chebbias}}
\label{proofCB}

For approaching the proposition \ref{Chebbias} we reformulate it in a simpler way as
\begin{proposition}\label{Regbias}
One introduces the regularized couting function $\pi'(x;q,l):=\pi(x;q,l)-\psi(x;q,l)/\log x$. The statement $\pi'(x;q,N)>\pi'(x;q,R),~ \forall x$ (ii),  is equivalent to GRH for the modulus $q$.
\end{proposition}
\begin{proof}
First observe that proposition \ref{Chebbias} follows from proposition \ref{Regbias}. This is straightforward because according to \cite[p. 260]{Robin84}, the prime number theorem for arithmetic progressions leads to the approximation 
\begin{equation}
\mbox{li}[\phi(q)\psi(x;q,l)]\sim \mbox{li}(x)+\frac{\phi(q)\psi(x;q,l)-x}{\log x}.
\label{equiv}
\end{equation}
As a result
\begin{eqnarray}
&\delta'(x,q)=B(x;q,R)-B(x;q,N)\nonumber\\
&=\mbox{li}[\phi(q)\psi(x;q,R)]-\mbox{li}[\phi(q)\psi(x;q,N)]+\phi(q)\delta(x,q)\nonumber\\
&\sim\phi(q)[\pi'(x;q,N)-\pi'(x;q,R)].\nonumber
\end{eqnarray}
The asymtotic equivalence in $(\ref{equiv})$ holds up to the error term \cite[p. 260]{Robin84} $O(\frac{R(x)}{x \log x})$, with
$$R(x)=\mbox{min} \left( x^{\theta_q}\log^2 x,xe^{-a\sqrt{\log x}}\right),~a>0,$$  
$$\theta_q=\mbox{max}_{\kappa \mod q}(\sup \Re (\rho),~ \rho ~\mbox{a}~\mbox{zero}~\mbox{of}~L(s,\kappa)).$$

Let us now look at the statement GRH $\Rightarrow (i)$.
Following \cite[p 178-179]{Rubin94}, one has
$$\psi(x;q,a)=\frac{1}{\phi(q)}\sum_{\kappa\mod q}\bar{\kappa}(a)\psi(x,\kappa)$$
and under GRH
$$\pi(x;q,a)=\frac{\pi(x)}{\phi(q)}-\frac{c(q,a)}{\phi(q)}\frac{\sqrt{x}}{\log x}+\frac{1}{\phi(q)\log x}\sum_{\kappa \ne \kappa_0}\bar{\kappa}(a)\psi(x,\kappa)+O(\frac{\sqrt{x}}{\log^2 x}),$$
where $\kappa_0$ is the principal character modulo $q$ and
$$c(q,a)=-1+\#\{1\le b\le q:b^2=a \mod q\}$$
for coprimes integers $a$ and $q$.
Note that for an odd prime $q=p$, one has $c(p,a)=\left(\frac{a}{p}\right).$

Thus, under GRH
\begin{eqnarray}
&\pi(x;q,N)-\pi(x;q,R)=\frac{1}{\phi(q) \log x}[\sqrt{x}(c(q,R)-c(q,N))\nonumber \\
&+\sum_{\kappa \mod q}
(\bar{\kappa}(N)-\bar{\kappa}(R))\psi(x,\kappa)+O(\left(\frac{\sqrt{x}}{\log^2x}\right)].
\label{piNR}
\end{eqnarray}
The sum could be taken over all characters because $\bar{\kappa}_0(N)=\bar{\kappa}_0(R)$.
In addition, we have
\begin{equation}
\psi(x;q,N)-\psi(x;q,R)=\frac{1}{\phi(q)}\sum_{\kappa \mod q}[\bar{\kappa}(N)-\bar{\kappa}(R)]\psi(x,\kappa).
\label{psiNR}
\end{equation}
Using (\ref{piNR}) and (\ref{psiNR}) the regularized bias reads
\begin{eqnarray}
&\delta'(x,q)\sim\pi'(x;q,N)-\pi'(x;q,R) \nonumber \\
&=\frac{\sqrt{x}}{\log x}[c(q,R)-c(q,N)]+O\left(\frac{\sqrt{x}}{\log^2 x}\right).
\label{regbias}
\end{eqnarray}
For the modulus $q=4$, we have $c(q,1)=-1+2=1$ and $c(q,3)=-1$ so that $\delta'(x,4)=\frac{2\sqrt{x}}{\log x}$ The same result is obtained for a prime modulus $q=p$ since $c(p,N)=-1$ and $c(p,R)=c(p,1)=\left(\frac{1}{p} \right)=1$.

This finalizes the proof that under GRH, one has the inequality  $\pi'(x;q,N)>\pi'(x;q,R)$.

If GRH does not hold, then using \cite[lemma 2]{Robin84}, one has
$$B(x;q,a)=\Omega_{\pm}(x^{\xi})~\mbox{for}~\mbox{any}~\xi<\theta_q.$$
Applying this assymptotic result to the residue classes $a=R$ and $a=N$, there exist infinitely many values $x=x_1$ and $x=x_2$ satisfying
$$B(x_1;q,R)<-x_1^{\xi}~\mbox{and}~B(x_2;q,N)>x_2^{\xi}~\mbox{for}~\mbox{any}~\xi<\theta_q,$$
so that one obtains 
\begin{equation}
B(x1;q,R)-B(x_2;q,N)<-x_1^{\xi}-x_2^{\xi}<0.
\label{in1}
\end{equation}
Selecting a pair $(x_1,x_2)$ either
$$B(x_1;q,R)>B(x_2;q,R)$$
so that $B(x_2;q,R)-B(x_2;q,N)<0$ and (i) is violated at $x_2$, or
\begin{equation}
B(x_1;q,R)<B(x_2;q,R) .
\label{in2}
\end{equation}
In the last case, either $B(x_1;q,N)>B(x_2;q,N)$, so that $B(x_1;q,R)-B(x_1;q,N)<0$ and the inequality (i) is violated at $x_1$,
or simultaneously
$$B(x_1;q,N)<B(x_2;q,N)~\mbox{and}~B(x_1;q,R)<B(x_2;q,R),$$
which implies $(\ref{in1})$ and the violation of (i) at $x=x_1=x_2$.
 
To finalize the proof of \ref{Regbias}, and simultaneously that of \ref{Chebbias}, one makes use of the asymptotic equivalence of (i) and (ii),
that is if GRH is true $\Rightarrow$ (ii) $\Rightarrow$ (i), and if GRH is wrong, (i) may be violated and (ii) as well.

Then, proposition \ref{Chebbiasp} also follows as a straigthforward consequence of proposition \ref{Chebbias}.

\end{proof}n 

\section{Summary}
We have found that the asymmetry in the prime counting function $\pi(x;q,a)$ between the quadratic residues $a=R$ and the non-quadradic residues $a=N$ for the modulus $q$ can be encoded in the function $B(x;q,a)$ [defined in (\ref{eq11})] introduced by Robin the context of GRH \cite{Robin84}, or into the regularized prime counting function $\pi'(x;q,a)$, as in Proposition \ref{Regbias}. The bias in $\pi'$ reflects the bias in $\pi$ conditionaly under GRH for the modulus $q$. Our conjecture has been initiated by detailed computer calculations presented in Sec. \ref{regularizedCB} and proved in Sec. \ref{proofCB}. Further work could follow the work about the connection of $\pi$, and thus of $\pi'$, to the sum of squares function $r_2(n)$ \cite{Moree2003}.

\end{document}